\documentclass{article}

\usepackage{amssymb, latexsym,pdfsync,amsmath,amsthm, ulem,hyperref,graphicx}
\usepackage{fullpage}
\usepackage{pgf,tikz}
\usepackage{mathrsfs}
\usetikzlibrary{arrows}
\newtheorem{theorem}{Theorem}
\newtheorem{corollary}{Corollary}
\newtheorem{lemma}{Lemma}
\newtheorem{proposition}{Proposition}

\theoremstyle{definition}
\newtheorem{definition}{Definition}

\newtheorem*{maintheorem}{Main Theorem}

\newcommand{\FF}{\mathbb{F}}

\newcommand{\Fq}{\mathbb{F}_q}
\newcommand{\Fqn}{\mathbb{F}_{q^n}}

\newcommand{\HH}{\mathcal H}

\newcommand{\D}{\mathcal D}

\newcommand{\C}{\mathcal C}
\newcommand{\cI}{\mathcal I}

\def\S{\mathcal{S}}

\def\Fq{{\mathbb{F}}_q}

\def\Aut{\mathrm{Aut}}

\def\PG{\mathrm{PG}}

\def\GL{\mathrm{GL}}

\def\dim{\mathrm{dim}}

\def\End{\mathrm{End}}

\newcommand{\rank}{\mathrm{rank}}

\begin{document}
\title{Binary additive MRD codes with minimum distance $n-1$ must contain a semifield spread set}
\author{John Sheekey
\footnote{School of Mathematics and Statistics,
University College Dublin,
Ireland; john.sheekey@ucd.ie}}
\maketitle

\begin{abstract}
In this paper we prove a result on the structure of the elements of an additive {\it maximum rank distance (MRD) code} over the field of order two, namely that in some cases such codes must contain a semifield spread set. We use this result to classify additive MRD codes in $M_n(\FF_2)$ with minimum distance $n-1$ for $n\leq 6$. Furthermore we present a computational classification of additive MRD codes in $M_4(\FF_3)$. The computational evidence indicates that MRD codes of minimum distance $n-1$ are much more rare than MRD codes of minimum distance $n$, i.e. semifield spread sets. In all considered cases, each equivalence class has a known algebraic construction.
\end{abstract}

\section{Introduction}


In this paper we consider the classification problem for optimal codes in the rank metric, known as maximum rank distance (MRD) codes. Such codes are of interest in part due to applications \cite{SiKsKo2008}, and due to recently observed connections to topics in finite geometry such as semifields (nonassociative division algebras) and scattered linear sets \cite{SheekeyMRD}, \cite{Willems}. Not much is known about the structure of the set of invertible elements in an MRD code; in this paper we address this question, prove the following main result. 

\begin{maintheorem}
For any additively closed set $\C\subset M_n(\FF_2)$ such that $|\C|= 2^{2n}$ and $\rank(A)\geq n-1$ for all $0\ne A\in \C$, there exist two presemifields $(\FF_2^n,\star)$ and $(\FF_2^n,\circ)$ such that
\[
\C = \{x\mapsto a\star x-b\circ x:a,b\in \FF_2^n\}.
\]
\end{maintheorem}

In other words, binary additive MRD code with minimum distance $n-1$ are spanned by two binary additive MRD codes with minimum distance $n$ (which are also known as semifield spread sets). 

We use this result to advance the computational classification of MRD codes, showing that all additive MRD codes in $M_n(\FF_2)$ with minimum distance $n-1$ for $n\leq 6$ are equivalent to Delsarte-Gabidulin codes, and all additive MRD codes in $M_n(\FF_3)$ with minimum distance $n-1$ for $n\leq 4$ are equivalent to a code with a known algebraic construction.

The results of this paper suggest some natural open problems, which we list in the summary.

\subsection{MRD codes and Semifields}

Detailed background on MRD codes can be found for example in \cite{SheekeyMRD}. We recall here the essential definitions. We denote the set of $m\times n$ matrices over the field $\FF_q$ as $M_{m\times n}(\Fq)$, and if $m=n$ we write $M_n(\Fq)$.

\begin{definition}
A set $\C\subset M_{m\times n}(\Fq)$ is said to be a {\it maximum rank-distance (MRD) code} if
\begin{itemize}
\item[(1)] $|\C| = q^{n(m-d+1)}$, and
\item[(2)] $\rank(A)\geq d$ for all $0\ne A\in \C$.
\end{itemize}
We call $d$ the {\it minimum distance} of $\C$. If $\C$ is additively closed, we say that $\C$ is an {\it additive MRD code}. If $q^{n(m-d)}<|\C| <q^{n(m-d+1)}$ we say that $\C$ is a {\it quasi-MRD code}.
\end{definition}

In the case of $n=m=d$, MRD codes correspond precisely to algebraic structures known as {\it (pre)quasifields}, and additive MRD codes correspond to {\it (pre)semifields}; in this setting such a set of matrices is more commonly known as a {\it spread set} and {\it semifield spread set} respectively. 

\begin{definition}
A {\it finite presemifield} is a finite-dimensional division algebra over a finite field in which multiplication is not assumed to be associative. A {\it finite semifield} is a finite presemifield containing a multiplicative identity element. The {\it semifield spread set} associated to a presemifield $\S = (\Fq^n,\circ)$ is the set $\C(\S) := \{x\mapsto x\circ y : y\in \Fq^n\}\subset M_n(\Fq)$. 
\end{definition}

\begin{definition}
Two presemifields $(\Fq^n,\circ)$ and $\Fq^n,\star)$ are {\it isotopic} if there exist invertible additive maps $A,B,C$ on $\Fq^n$ such that
\[
A(x\circ y) = B(x)\star C(y).
\]
Two presemifields are {\it transpose-isotopic} if one is isotopic to the transpose of the other, where the transpose of a semifield is as defined in \cite{Knuth}. 
\end{definition}

It is well known that every presemifield is isotopic to a semifield \cite{LaPo2011}. The transpose of a semifield is one element of the {\it Knuth orbit}, which is a set of up to six isotopy classes. However, as only the transpose operation generalises to arbitrary MRD codes, we will not consider this here. We refer also to \cite{LaPo2011} for further information.

\begin{definition}
Two sets $\C_1,\C_2 \subset M_n(\Fq)$ are {\it equivalent} if there exist $A,B\in \GL(n,q)$, $\rho \in \Aut(\Fq)$, such that
\[
\C_1 = \{AX^\rho B:X\in \C_2\},
\]
or 
\[
\C_1 = \{A(X^T)^\rho B:X\in \C_2\}.
\]
The {\it left-} and {\it right-idealiser} of a set $\C\subset M_n(\Fq)$ are the sets $\cI_\ell(\C) = \{A:A\C\subset\C\}$ and $\cI_r(\C) = \{A:\C A\subset\C\}$ respectively. The {\it automorphism group} is $\Aut(\C) = \{(A,B):A\C B^T=\C\}\subset \GL(n,q)\times \GL(n,q)$. 
\end{definition}

Much of the focus on MRD codes to date has been on codes which are {\it $\Fqn$-linear}; that is, codes in which the left idealiser contains a field isomorphic to $\Fqn$. In the literature on semifields, the idealisers coincide with some of the {\it nuclei} \cite{MaPoNuc}.

The following is well known \cite{LaPo2011}.
\begin{theorem}
For any finite presemifield $\S$, the set $\C(\S)$ is an additive MRD code with minimum distance $n$. Every additive MRD code with minimum distance $n$ in $M_n(\Fq)$ is the semifield spread set of some finite presemifield. Two presemifields are isotopic or transpose-isotopic if and only if their semifield spread sets are equivalent
\end{theorem}

MRD codes exist for all parameters, due to a construction of Delsarte \cite{Delsarte1978}; see Section \ref{sec:construct}. 

The problem of classifying semifields, or in bounding the number of {\it isotopy classes} of semifields, has been studied by many authors. Classifications have been performed for $2^4,2^5,2^6,3^4,3^5,5^4,7^4$. See below table for data and references. Futhermore Kantor \cite{Kantor2006} conjectured that the number of isotopy classes of semifields grows exponentially.

The problem of classifying MRD codes for other minimum distances has received much less attention, and thus there is less of a consensus on the ubiquity of such codes. A surprising recent result \cite{Hein2017} showed that in some cases, MRD codes can be rare; in particular, there is only one MRD code of minimum distance $3$ in $M_{4\times 4}(\FF_2)$, which is additive (and a Delsarte-Gabidulin code). For comparison, there are eight equivalence classes of MRD codes of minimum distance $4$ in $M_{4\times 4}(\FF_2)$, including three additive codes \cite{Knuth}, \cite{DempTP16}.

In this paper we show a theoretical result regarding subcodes of additive MRD codes over $\FF_2$ in Section \ref{sec:main}, which, given the known computational results on semifields, drastically reduces the amount of computation needed in order to classify additive MRD codes of minimum distance $n-1$ in $M_{n\times n}(\FF_2)$. In Section \ref{sec:comp} we present the results of some computations, which fully classify all additive MRD codes of minimum distance $n-1$ in $M_{n\times n}(\FF_2)$ for $n=5,6$. All examples turn out to be Delsarte-Gabidulin codes. Furthermore in Section \ref{sec:odd} we classify all additive MRD codes in $M_4(\FF_3)$; all examples turn out to contain a semifield spread set, and belong to one of three known families.

\subsection{Partial Spreads}

We recall some definitions from finite geometry, which will be crucial to our main result.

\begin{definition}
A {\it partial $t$-spread} of a vector space $V(N,q)$ is a set of $t$-dimensional subspaces of $V(N,q)$ which pairwise intersect in the trivial space of dimension zero.
\end{definition}

Clearly for any partial $t$-spread $\D$ of $V(N,q)$ we have that $|\D|\leq \frac{q^N-1}{q^t-1}$. In the case of equality, $\D$ is called a {\it spread}, and every nonzero vector of $V(N,q)$ appears in precisely one element of $\D$. Equality is possible if and only if $t$ divides $N$, a well-known result due to Segre \cite{Segre}. If $\D$ is not contained in any strictly larger partial $t$-spread, then $\D$ is said to be {\it maximal}. Note that some authors instead use the term {\it complete}.

It is well-known that spread sets, i.e. MRD codes in $M_n(\Fq)$ with minimum distance $n$, are in one-to-one correspondence with $n$-spreads of $V(2n,q)$. Furthermore, general MRD codes can be {\it lifted} to so-called {\it subspace codes}. However in this paper we will associate a partial spread to an MRD code in a different way in Section \ref{sec:mrdspread}.

\subsection{Subspace Codes}

Some of the recent interest in rank-metric coding is due to the following connection with subspace codes. We recall briefly this connection, and refer to \cite{SiKsKo2008} for further detail.

Given a matrix $A\in M_{m\times n}(\Fq)$, we can define an $m$-dimensional subspace
\[
U_A = \{(x,xA):x\in \FF_{q^m}\} \leq V(n+m,q).
\]
Then $\dim(U_A\cap U_B) = 2\rank(A-B)$, and thus MRD codes give rise to (constant dimension) subspace codes with high minimum distance; we refer to this as a {\it lifted MRD code}. In the case $m=n=d(\C)$, MRD codes (spread sets) are in one to one correspondence with spreads. If $d(\C)<n=m$, it is not necessarily true that MRD codes give best possible subspace codes; see for example \cite{EtzionLifted}, \cite{CoPa}. However, in \cite{Hein2017} it was shown that the lifted Delsarte-Gabidulin code in $M_4(\FF_2)$ with minimum distance $3$ is in fact the best possible subspace code with its parameters. This further illustrates the requirement for new results on binary MRD codes.

\section{Known Constructions}
\label{sec:construct}

The following are the known constructions for additive MRD codes in $M_n(\Fq)$. We note that there are other constructions for rectangular matrices, but we do not consider those here.

\begin{proposition}
Suppose $\sigma$ is an $\Fq$-automorphism of $\Fqn$, $\phi_1,\phi_2$ are two additive functions from $\Fqn$ to itself, and $k\leq n-1$. Let $\HH_k(\phi_1,\phi_2) $ be the set of $\Fq$-linear maps
\[
\HH_k(\phi_1,\phi_2) = \left\{x\mapsto \phi_1(a)x+\left(\sum_{i=1}^{k-1}f_i x^{\sigma^i}\right)+\phi_2(a)x^{\sigma^k}:a,f_i\in\Fqn\right\}\subset \End_{\Fq}(\Fqn)\simeq M_n(\Fq).
\]
If $N(\phi_1(a))\ne (-1)^{nk}N(\phi_2(a))$ for all $a\in \Fqn^*$, then $\HH_k(\phi_1,\phi_2)$ is an MRD code.
\end{proposition}
To date this result has been used to construct the following. Here $s$ is any integer relatively prime to $n$. Note that what we designate as Delsarte-Gabidulin codes include the generalised Gabidulin codes; we do not distinguish between them here. 
\begin{center}
\begin{tabular}{|c|c|c|c|c|c|c|}
\hline
Name &$\sigma$&$\phi_1(a)$&$\phi_2(a)$&Conditions&Reference\\
\hline
DG&$q^s$&$a$&$0$&$-$&\cite{Delsarte1978}, \cite{GGab}\\
TG&$q^s$&$a$&$\eta a^{p^h}$&$N_p(\eta)\ne(-1)^{nk}$&\cite{SheekeyMRD},\cite{LuTrZh2015},\cite{Ozbudak1}\\
TZ&$q^s$&$a_0$&$\eta a_1$&$n$ even, $N(\eta)\notin \Box$&\cite{TrZhHughes}\\
\hline
\end{tabular}
\end{center}
DG = Delsarte-Gabidulin, TG=Twisted Gabidulin, TZ = Trombetti-Zhou

Further recent constructions in the case $m=n\in \{6,7,8\}$ can be found in \cite{CsMaPoZa}, \cite{CsMaZu}, \cite{CsMaPoZh}. A construction generalising the Twisted Gabidulin codes to various new parameters (that is, new orders of nuclei/idealisers) can be found in \cite{SheekeySkew}; these codes lie in $M_n(\FF_{q^s})$ for some $s>1$, and thus do not give examples in the cases considered in this paper. 

There are many constructions for nonlinear MRD codes, for example \cite{CoMaPa}; we do not list them all here, as we do not prove any results on this topic.

Note that the only constructions which are valid in $M_n(\FF_2)$ are the Delsarte-Gabidulin codes. We will see shortly that for small values of $n$, no other codes exist over $\FF_2$.  

\section{Known Classifications}
\subsection{Square matrices}

Dickson \cite{Dickson1905} was the first to show the existence of proper finite semifields. Furthermore he showed that proper semifields must be at least three-dimensional algebras.
\begin{theorem}[Dickson]
Every finite semifield of order $q^2$ with centre $\Fq$ is isotopic to a field.
\end{theorem}
Paraphrasing this into the language of MRD codes gives the following.
\begin{corollary}
Every $\Fq$-linear MRD code $\C\subset M_2(\Fq)$ is equivalent to a Delsarte-Gabidulin code. If $d(\C)=2$ then $\C$ is equivalent to $\C(\FF_{q^2})$. If $d(\C)=1$, then $\C$ is the trivial code $M_2(\Fq)$.
\end{corollary}
Additive MRD codes in $M_2(\Fq)$ which are not $\Fq$-linear correspond to proper semifields two-dimensional over a nucleus. Classification has been performed for the case $q=p^2$ for $p$ prime, and partial classification for the case $q=p^3$. There are various constructions known, we refer to \cite{LaPo2011}.

The following result is due to Menichetti \cite{Menichetti1977}.
\begin{theorem}[Menichetti]
Every finite semifield of order $q^3$ with centre $\Fq$ is isotopic to a field or a generalised twisted field.
\end{theorem}
Paraphrasing this into the language of MRD codes, and exploiting the duality result of Delsarte, gives the following.
\begin{corollary}
Every $\Fq$-linear MRD code in $M_3(\Fq)$ is equivalent to a Delsarte-Gabidulin code.
\end{corollary}


The smallest open cases are $\FF_3$-subspaces of $M_3(\FF_9)$ of minimum distance $3$; $\FF_3$-subspaces of $M_2(\FF_{27})$ of minimum distance $3$; $\FF_3$-subspaces of $M_4(\FF_3)$ of minimum distance $3$; $\FF_2$-subspaces of $M_5(\FF_2)$ and $M_6(\FF_2)$. We will complete the last two of these in this paper.

The following table lists the previously known computer classifications for MRD codes of square matrices. New results in this category can be found in Section \ref{sec:summary}. Here we list the available data in terms of isotopy classes and Knuth orbits; later we will write in terms of equivalence as defined above.

\begin{center}
\begin{tabular}{|c|c|c|c|c|}
\hline
$n \times m$ &$q$&$d$&\#Isotopy Classes&Reference\\
&&& (Knuth Orbits)&\\
\hline
$4\times 4$&2&$4$&3(3)&\cite{Knuth}\\
$4\times 4$&2&$3$&1&\cite{SheekeyMRD},\cite{Hein2017}\\
$4\times 4$&3&$4$&27 (12)&\cite{DempSem81}\\
$4\times 4$&4&$4$&(28)&\cite{Rua2011a}\\
$4\times 4$&5&$4$&(42)&\cite{Rua2011a}\\
$4\times 4$&7&$4$&(120)&\cite{Rua2012}\\
\hline
$5\times 5$&2&$5$&6 (3)&\cite{Walker}\\
$5\times 5$&3&$5$&23 (9)&\cite{Rua2011b}\\
\hline
$6\times 6$&2&$6$&332 (80)&\cite{Rua2009}\\
\hline
\end{tabular}
\end{center}

\subsection{Rectangular matrices}

The problem of classifying $\Fq$-linear MRD codes of rectangular matrices with maximal possible minimum distance can be translated into classifying quasi-MRD codes of square matrices. We do not claim that this result is new, but we include a proof for completeness. The idea is essentially that of \cite{MLIrsee}, \cite{LaSh233}.

\begin{theorem}
\label{thm:tensor}
Equivalence classes of $\Fq$-linear MRD codes in $M_{m\times n}(\Fq)$ with minimum distance $m$ are in one-to-one correspondence with equivalence classes (excluding transposition) of $m$-dimensional subspaces of $M_{n\times n}(\Fq)$ in which every nonzero element is invertible.
\end{theorem}

\begin{proof}
Given an $n$-dimensional $\Fq$-subspace $\C$ of $M_{m\times n}(\Fq)$, we can define an $m\times n\times n$ tensor, or equivalently a trilinear form, as follows. Choose a basis $\{E_1,\cdots,E_n\}$ for $\C$, and define a trilinear form
\[
T(w,u,v) = w\left(\sum_i v_i E_i\right)u^T.
\]
 As $\C$ is an MRD code, this implies that for all $0\ne a\in \Fq^m, 0\ne c\in \Fq^n$, the map $u \mapsto T(a,u,c)$ is not identically zero. 

Consider now the set of bilinear forms $T_a:(u,v)\mapsto T(a,u,v)$, which can be viewed as an $m$-dimensonal subspace of $M_n(\Fq)$. Then each $T_a$ is invertible, for otherwise there would exists $0\ne c\in \Fq^n$ with $T(a,u,c)=0$ for all $u\in \Fq^n$, a contradiction. Hence $\{T_a:a\in \Fq^m\}$ gives rise to an $m$-dimensional subspace of $M_n(\Fq)$ in which every nonzero element is invertible. The converse is similar.

The equivalence follows easily from the definition of equivalence of a tensor; see for example \cite{MLIrsee}.
\end{proof}

Classification of $3$-dimensional $\Fq$-subspaces of $M_{2\times 3}(\Fq)$ was done in \cite{LaSh233}. This implies the following result.
\begin{theorem}[\cite{LaSh233}]
There is precisely one equivalence class of MRD codes in $M_{2\times 3}(\FF_q)$ with minimum distance $2$.
\end{theorem}

Classification of MRD codes in $M_{3\times 4}(\FF_2)$ was done in \cite{Kurz}. 
\begin{theorem}[\cite{Kurz}]
There are $37$ equivalence classes of MRD codes in $M_{3\times 4}(\FF_2)$ with minimum distance $3$. There are $7$ equivalence classes of additive MRD codes in $M_{3\times 4}(\FF_2)$ with minimum distance $3$, and by duality, $7$ equivalence classes of additive MRD codes in $M_{3\times 4}(\FF_2)$ with minimum distance $2$.
\end{theorem}

Computer searches performed using MAGMA give the following. The search was made much more efficient by exploiting Theorem \ref{thm:tensor}, and using techniques developed in the classification of semifields (see for example \cite{Rua2009}). 

\begin{theorem}
There are $43$ equivalence classes of additive MRD codes in $M_{3\times 4}(\FF_3)$ with minimum distance $3$, and, by duality, $43$ equivalence classes of additive MRD codes in $M_{3\times 4}(\FF_3)$ with minimum distance $2$.
\end{theorem}

This calculation took approximately 7.5 hours on a single CPU. Note that it is actually quicker to classify MRD codes in $M_{4\times 4}(\FF_3)$ with minimum distance $4$; the author's calculation takes less than one hour.

%

\subsection{Computer Classifications}
The following table lists the known computer classifications for MRD codes of rectangular matrices. Results without reference were calculated by the author using MAGMA.

\begin{center}
\begin{tabular}{|c|c|c|c|c|}
\hline
$n \times m$ &$q$&$d$&\# &Reference\\
\hline
$3\times 4$&2&$3$&7&\cite{Kurz}\\
\hline
$3\times 4$&3&$3$&43&\\
\hline
$3\times 5$&2&$3$&368&\\
\hline
$4\times 5$&2&$4$&2678&\\
\hline
$3\times 6$&2&$3$&95877&\cite{Rua2009}\\
\hline
\end{tabular}
\end{center}

Note that there are many more classes of additive MRD codes in the non-square case. This is perhaps not too surprising, as many inequivalent rectangular MRD codes can be obtained from a single equivalence class of square MRD codes. See for example \cite{CsSi}, \cite{SchmidtPunc}.

\section{Structure of elements of smallest rank in additive MRD codes}
\label{sec:mrdspread}

Consider an additive MRD code $\C$ in $M_{m\times n}(\Fq)$, $m\leq n$, with minimum rank-distance $d$. Note that $\dim(\C)=n(m-d+1)$. It was shown in for example \cite[Theorem 2]{DuGoMcGSh} that the elements of minimal rank are partitioned by subspaces, as follows. For any $(m-d)$-dimensional space $U$ of $\Fq^m$, define 
\[
\C_U := \{X:X\in \C, U\leq \ker(X)\}.
\]
Let $\C_d$ denote the set of elements of $\C$ of rank $d$. Then $\dim(\C_U) = n$, $\dim(\C_U\cap \C_W)=0$ for $U\ne W$, and
\[
\C_d = \bigcup_{\dim(U)=m-d} \C_U^\times.
\]
Thus we have that $\D_\C := \{\C_U:\dim(U)=m-d\}$ is a partial $n$-spread of $\C$. Note that $|\D_\C| = {m \brack d}$.

We would like to determine whether $\C$ must contain an additive MRD-code with minimum distance $d+1$. We summarise with the following lemma.

\begin{lemma}\label{lem:key}
If there exists an additive MRD code $\C$ in $M_{m\times n}(\Fq)$, $m\leq n$, with minimum rank-distance $d$, then there exists a partial $n$-spread $\D$ of $V(n(n-d+1),q)$ of size ${m \brack d}$. If $\C$ contains an additive MRD code with minimum distance $d+1$, then there exists a subspace of $V(n(n-d+1),q)$ of dimension $n(n-d)$ meeting every element of $\D$ trivially.
\end{lemma}

\subsection{Main theoretical result: the case $d=m-1$, $n=m$}
\label{sec:main}

In this case, $\D_\C$ is a partial $n$-spread of the $2n$-dimensional space $\C$. By Lemma \ref{lem:key}, $\C$ contains a semifield spread set if and only if $\D_\C$ is not a maximal spread. By \cite{Bruen} we have the following.

\begin{theorem}[Bruen]
A maximal partial $n$-spread $\D$ of $V(2n,q)$ satisfies
\[
|\D|\leq q^n - \sqrt{q}.
\]
\end{theorem}
If $q=2$, then $|\D_\C| = 2^m-1$. Hence if $n=m$, $\D_\C$ cannot be maximal; indeed, it can be extended to a spread of $\C$. Hence $\C$ contains $n$-dimensional spaces disjoint from $\C_{n-1}$, which are therefore semifield spread sets. In fact, there are precisely two such spaces contained in $\C$. Hence we have the following.

\begin{theorem}
\label{thm:main}
A binary additive MRD code with minimum distance $n-1$ contains a binary additive MRD code with minimum distance $n$. For any such MRD code $\C$ there exist two presemifields $(\FF_2^n,\star)$ and $(\FF_2^n,\circ)$ such that
\[
\C = \{x\mapsto a\star x-b\circ x:a,b\in \FF_2^n\}.
\]
\end{theorem}


\section{Computational Results}
\label{sec:comp}

The result of the preceding section can be used in order to classify binary additive MRD codes very efficiently. All computations below were carried out using the computer algebra package MAGMA \cite{MAGMA}.

\subsection{$q=2$, $n=m=4$}

For $n=m=4$, there are precisely three additive MRD codes with minimum distance $4$ \cite{Knuth}. It was noted in \cite{SheekeyMRD} that only one of these can be extended to an additive MRD code with minimum distance $3$, and this extension is unique. In fact, a much stronger result was proved in \cite{Hein2017}, where it was shown that there is a unique (not necessarily additive) MRD code in $M_4(\FF_2)$ with minimum distance $3$.

\begin{theorem}[\cite{Hein2017}]
There exists precisely one MRD code in $M_4(\FF_2)$ with minimum distance $3$.
\end{theorem}

\subsection{$q=2$, $n=m=5$}

For $n=m=5$, there are six additive MRD codes with minimum distance $5$. These were classified by Walker \cite{Walker}, and can be found online at \cite{DempData}. A computer search returns that only one of these, that being the MRD code corresponding to the field $\FF_{32}$, can be extended to an MRD code with $d=4$. There are two equivalence classes of these codes; both are generalised Delsarte-Gabidulin codes.

\begin{theorem}
There exists precisely two equivalence classes of additive MRD codes in $M_5(\FF_2)$ with minimum distance $4$. Both of these belong to the family of generalised Delsarte-Gabidulin codes. 
\end{theorem}


Four of the other semifields can be extended to a 6-dimensional space with minimum distance 4, but not a 7-dimensional space. There are 5 equivalence classes of 6-dimensional quasi-MRD codes containing $S_i$ for $i=1,2,3,6$. None of these coincide.

There are no 6-dimensional quasi-MRD codes containing $S_4$.

Therefore there are exactly 24 equivalence classes of 6-dimensional quasi-MRD codes, 20 of which are not extendable.

\begin{center}
\begin{tabular}{|c|c||c|c|c|}
\hline
Dim&\# &$\C(\FF_{32})$&$\C(S_i)$&$\C(S_4)$\\
\hline
5&6&1&1&1\\
6&24&4&5&0\\
7&4&4&0&-\\
8&4&4&-&-\\
9&4&4&-&-\\
10&2&2&-&-\\
\hline
\end{tabular}
\end{center}
This table shows the number of equivalence classes of linear (quasi-)MRD codes, and the number containing each semifield spread set. Here the index $i$ is any of $1,2,3,6$.

\subsection{$q=2$, $n=m=6$}

Of the 332 semifields (in 180 equivalence classes, 80 Knuth orbits) of order $2^6$, which are listed in \cite{Rua2009}, only the spread set $\C(\FF_{2^6})$ of the finite field $\FF_{2^6}$ can be extended to an additive MRD code with minimum distance $5$. There is a unique such code, which is a Delsarte code. Note that the generalised Gabidulin codes in this case are all equivalent to the Delsarte code.

\begin{theorem}
There exists precisely one equivalence classes of additive MRD codes in $M_6(\FF_2)$ with minimum distance $5$. This belongs to the family of Delsarte-Gabidulin codes.
\end{theorem}

\section{Odd Characteristic}
\label{sec:odd}

Theorem \ref{thm:main} does not necessarily hold for $q>2$. However we can still consider the problem of classifying additive MRD codes containing a semifield spread set. We do so for the case of $M_4(\FF_3)$; in fact, this turns out to be no restriction, due to the following computational result. 

\begin{theorem}
Every additive MRD codes in $M_4(\FF_3)$ with minimum distance $3$ contains a semifield spread set.
\end{theorem}

This result was calculated by taking representatives of each of the $43$ equivalence classes of MRD codes with minimum distance $3$ in $M_{3\times 4}(\FF_3)$, adding a row of zeroes, and attempting to extend to an additive MRD code in $M_4(\FF_3)$, similar to what was described in \cite{Kurz}. The classification in $M_{3 \times 4}(\FF_3)$ was carried out using Theorem \ref{thm:tensor}. It turns out that only five of these representatives can be extended. In each case it was found that the obtained MRD codes contained semifield spread sets. The whole classification took approximately 84 hours of computation time. We describe the results of this computation now.

\subsection{$q=3$, $m=n=4$, $d=3$}

The known constructions are: Delsarte-Gabidulin (DG); Twisted Gabidulin (TG); Trombetti-Zhou (TZ). 

\begin{theorem}
There exist precisely five equivalence classes of additive MRD codes in $M_4(\FF_3)$ with minimum distance $3$ containing a semifield spread set. These belongs to the following families: Delsarte-Gabidulin; twisted Gabidulin; and Trombetti-Zhou. 
\end{theorem}

There are 27 isotopism classes of semifields of order $81$, falling into 12 Knuth orbits; see \cite{DempSem81}, with representatives available at \cite{DempData}. These give $20$ equivalence classes of additive MRD codes in $M_4(\FF_3)$ with minimum distance $4$.

The only semifield spread sets which can be extended to additive MRD codes with minimum distance $3$ belong to the Knuth orbit of the semifields denoted $III,VI,VII,IX,X,XI,XII$ in \cite{DempSem81}. We will denote these by $\S_i$, where $i$ is the integer corresponding to the Roman numeral used by Dempwolff. 

These semifields belong to the following families:  $\S_6$ is Dickson-Knuth Type 2;  $\S_7$ is a Boerner-Lantz semifield, with one nucleus of order 9; $\S_9$ is a generalised twisted field, with one nucleus of order 9; $\S_{10}$ is a cyclic semifield, with all nuclei of order 9; $\S_{11}$ is a cyclic semifield, with all nuclei of order 9; $\S_{12}$ is the field $\FF_{81}$. The semifield $\S_3$ does not belong to a known construction. The below table shows the semifield spread sets contained in each additive MRD codes. Here $t$ denotes transpose, and $d$ denotes the (semifield) dual operation; note that this is distinct from the (Delsarte) dual operation on MRD codes.

\begin{center}
\begin{tabular}{|c|c|c|c||c|c|c|c|c|c|c|}
\hline
Name&Family&$[\#\cI_\ell,\#\cI_r]$ &$\# \Aut(\C)$&$\C(\FF_{81})$&$\C(\S_3)$&$\C(\S_6)$&$\C(\S_7)$&$\C(\S_9)$&$\C(\S_{10})$&$\C(\S_{11})$\\
\hline
A&TG&$[3,3]$&640&I&-&I&-&I&-&-\\
B&TG&$[3,3]$&640&I&d&-&-&I&-&-\\
C&TZ&$[9,9]$&1024&I&-&-&d,td&d,td&I&I\\
D&TG&$[81,9]$&1280&I&-&-&-&-&I&I\\
E&DG&$[81,81]$&25600&I&-&-&-&d,td&-&I\\
\hline
\end{tabular}
\end{center}


Recall here that equivalence includes the transpose/adjoint operation, and note that the effect of transposition on the idealisers is to interchange the left and right idealisers.

\section{Summary and Open Problems}
\label{sec:summary}

The following table summarises the known classifications of additive MRD codes in $M_n(\Fq)$ with minimum distance $d$ at least $n-1$. We note that, while the number of classes of semifields seems to grow quickly, the number of additive MRD codes with minimum distance $n-1$ appears to remain low, and all examples fall into known families.



\begin{center}
\begin{tabular}{|c|c|c|c||c|c|}
\hline
q&n&d&Equiv&Iso&Knuth\\
\hline
2&4&4&3&3&3\\
2&4&3&1&-&-\\
\hline
2&5&5&3&6&3\\
2&5&4&2&-&-\\
\hline
2&6&6&180&332&80\\
2&6&5&1&-&-\\
\hline
\hline
3&4&4&20&27&12\\
3&4&3&5&-&-\\
\hline
\end{tabular}
\end{center}

Note that there are $332$ isotopy classes of semifields of order $64$, falling into $80$ Knuth orbits. However as the $S_3$-action of the Knuth orbit does not generalise to all MRD codes, we present instead the number of equivalence classes as defined in this paper. We note that this is the number of semifields up to isotopy and transposition. Similarly, there are $27$ isotopy classes of semifields in $12$ Knuth orbits.

Here we present some natural open questions suggested by the results of this paper.
\begin{itemize}
\item
Do there exist any proper non-additive MRD codes for $q=2$, $n=5,6$, $d=n-1$?
\item
Do there exist any additive MRD codes for $q=2^e$, $e>1$, not containing a semifield spread set? 
\item
Do there exist any additive MRD codes not containing a semifield spread set for $q$ odd? The smallest open case is $q=3$, $n=4$, $d=3$.
\item
Are additive MRD codes with minimum distance $n-1$ always less plentiful than semifields?
\item
Are all additive MRD codes with minimum distance $n-1$ equivalent to one of the known constructions?
\item
Do all additive MRD codes with minimum distance $d<n$ contain an additive MRD code with minimum distance $d+1$?
\item
Do all additive MRD codes with minimum distance $d<n$ contain a code equivalent to $\C(\Fqn)$?
\end{itemize}

%
%

\end{document}